\documentclass[12pt,a4paper]{amsart}

\usepackage{amssymb,mathtools}
\usepackage{graphicx}
\usepackage{xcolor}
\usepackage{orcidlink}
\usepackage{hyperref}
\hypersetup{colorlinks=true,allcolors=blue}
\usepackage[margin=2.5cm]{geometry}

\newtheorem{theorem}{Theorem}[section]
\newtheorem{proposition}[theorem]{Proposition}
\newtheorem{corollary}[theorem]{Corollary}
\newtheorem{lemma}[theorem]{Lemma}

\theoremstyle{definition}
\newtheorem{example}[theorem]{Example}
\newtheorem{remark}[theorem]{Remark}

\newcommand{\II}{\mathbb{I}}
\newcommand{\RR}{\mathbb{R}}
\newcommand{\CC}{\mathcal{C}}
\newcommand{\de}{\,\mathrm{d}}
\newcommand{\Om}{\Omega_{\beta,\phi,\tau}}
\DeclareMathOperator{\sgn}{sgn}

\title[The exact region determined by $\beta$, $\phi$ and $\tau$]{The exact region determined by Blomqvist's beta, Spearman's footrule and Kendall's tau}

\author[]{Liljana Ferbar Tratar \orcidlink{0009-0001-1523-7137} }
\address{University of Ljubljana, School of Economics and Business, Ljubljana, Slovenia}
\email{Liljana.Ferbar.Tratar@ef.uni-lj.si}

\author[]{Damjana Kokol Bukovšek \orcidlink{0000-0002-0098-6784} }
\address{University of Ljubljana, School of Economics and Business, and Institute of Mathematics, Physics and Mechanics, Ljubljana, Slovenia}
\email{Damjana.Kokol.Bukovsek@ef.uni-lj.si}

\author[]{Blaž Mojškerc \orcidlink{0000-0001-8096-355X} }
\address{University of Ljubljana, School of Economics and Business, and Institute of Mathematics, Physics and Mechanics, Ljubljana, Slovenia}
\email{Blaz.Mojskerc@ef.uni-lj.si}

\author[]{Jacob Israel Orenday Lares \orcidlink{0009-0006-0872-5606} }
\address{Independent researcher, Cologne, Germany}
\email{Jacob.Orenday@gmail.com}

\author[]{Marcus Rockel \orcidlink{0009-0005-1818-3831} }
\address{University of Freiburg, Department of Quantitative Finance,
Institute for Economics, Freiburg, Germany}
\email{Marcus.Rockel@finance.uni-freiburg.de}

\author[]{Nik Stopar \orcidlink{0000-0002-0004-4957} } 
\address{University of Ljubljana, Faculty of Civil and Geodetic Engineering, University of Ljubljana, Faculty of Mathematics and Physics, and Institute of Mathematics, Physics and Mechanics, Ljubljana, Slovenia}
\email{Nik.Stopar@fgg.uni-lj.si}


\begin{document}

\begin{abstract}
We determine the exact region $\Omega_{\beta,\phi,\tau}= \{(\beta(C),\phi(C),\tau(C))\mid C\in\mathcal{C}\}$ of possible joint values of Blomqvist's beta, Spearman's footrule and Kendall's tau over the class $\mathcal{C}$ of all bivariate copulas.
The region consists precisely of all triples $(\beta,\phi,\tau)$ satisfying $-1\le \beta\le 1$, $\frac{3}{16}(1+\beta)^2-\frac12\le \phi\le 1-\frac38(1-\beta)^2$ and $\frac43 \phi-\frac13\le \tau\le \frac23 \phi+\frac13$.
In other words, the known exact $(\beta,\phi)$- and $(\phi,\tau)$-regions already characterize the joint region, so that, once the value of Spearman's footrule is fixed, Blomqvist's beta imposes no additional restriction on the possible values of Kendall's tau.
The proof is constructive: two $2$-parameter families of shuffles of $M$ realize all the points in $\Omega_{\beta,\phi,\tau}$ with $\phi = \frac{3}{16}(1+\beta)^2-\frac12$, and ordinal sums spread realizations of these two families over the whole region $\Omega_{\beta,\phi,\tau}$.
\end{abstract}

\maketitle

\section{Introduction}

Rank correlation coefficients are among the most widely used summaries of dependence between two random quantities, because they are invariant under strictly increasing transformations of either margin and require no moment assumptions \cite{Scar,Nels}.
In practice, several such coefficients are often estimated from the same data set, either to narrow down a suitable dependence model or to capture different aspects of association \cite{GeFaBeJa07}.
At the population level, a joint attainable region records exactly which coefficient combinations can arise from a common dependence structure \cite{ScPaTr,KoBuMo,KoBuMo26}.
The study of exact regions goes back to the classical inequalities between Kendall's tau and Spearman's rho \cite{Dani,DuSt} and has remained active ever since \cite{KoBuSt,KoBuLaMoSt,AnRo26}.
Here we settle the problem of the exact region for Blomqvist's beta, Spearman's footrule and Kendall's tau, for which the answer takes a particularly simple form.

A (bivariate) \emph{copula} is a distribution function on $\II^2$ with uniform marginals, where $\II\coloneqq[0,1]$. By Sklar's theorem, see \cite{Sklar}, copulas couple multivariate distribution functions with their one-di\-men\-sional margins, and the copula is unique whenever the margins are continuous; see \cite{Nels,DuSe} for an overview.
We denote by $\CC$ the set of all bivariate copulas, by $M(u,v)\coloneqq \min\{u,v\}$ the comonotonicity copula and by $W(u,v)\coloneqq \max\{0,u+v-1\}$ the countermonotonicity copula.
Measures of concordance, formalized axiomatically by Scarsini in \cite{Scar}, summarize on the scale $[-1,1]$ the degree of positive dependence captured by a copula.
The most prominent of the three coefficients studied in this paper is Kendall's tau, introduced in \cite{Ke38}, which is the probability of concordance minus the probability of discordance of two independent copies of the underlying random vector and, in copula form, becomes
\begin{equation}\label{eq:tau-def}
\tau(C)\coloneqq 4\int_{\II^2} C(u,v)\de C(u,v)-1.
\end{equation}
Spearman's footrule, introduced in \cite{Sp06} as an easily computed companion to Spearman's rank correlation, averages absolute rank differences; in terms of the diagonal section $\delta_C(u)\coloneqq C(u,u)$ it can be written as
\begin{equation}\label{eq:phi-def}
\phi(C)\coloneqq 6\int_0^1 \delta_C(u)\de u-2.
\end{equation}
Blomqvist's beta, introduced in \cite{Bl50} and also known as the medial correlation coefficient, compares the probability mass of the four quadrants determined by the medians and is given by
\begin{equation}\label{eq:beta-def}
\beta(C)\coloneqq 4C\bigl(\tfrac12,\tfrac12\bigr)-1.
\end{equation}
Kendall's tau and Blomqvist's beta are concordance measures in the sense of Scarsini \cite{Scar}, while Spearman's footrule is only a weak concordance measure in the sense of Liebscher \cite{Lieb} with range $[-\frac12,1]$; see \cite{GeNeBE10} for a review of the footrule.

The exact form of the mutual constraints between such coefficients has received considerable attention.
Schreyer et al.\ determined the exact $(\tau,\rho)$-region of Kendall's tau and Spearman's rho in \cite{ScPaTr}, settling a long-standing open question.
The exact regions of Spearman's footrule and Gini's gamma with respect to Blomqvist's beta were obtained in \cite{KoBuKoMoOm21}, and the regions of Kendall's tau with respect to Spearman's footrule and to Gini's gamma in \cite{KoBuSt}.
Beyond feasibility, exact regions have also been used to quantify how strongly knowledge of one coefficient restricts another: the concordance similarity measure was introduced in \cite{KoBuMo} and subsequently applied to Kendall's tau in \cite{KoBuSt}.
Recently, to the best of our knowledge, the first exact regions determined by three coefficients simultaneously have been described: the $(\beta,\phi,\gamma)$-region in \cite{KoBuMo26} and the $(\phi,\gamma,\tau)$-region in \cite{KoBuLaMoSt}.

In this paper we determine the exact region
\[
\Om\coloneqq \bigl\{(\beta(C),\phi(C),\tau(C))\mid C\in\CC\bigr\}\subseteq\RR^3.
\]
Throughout, we use the boundary functions
\begin{equation} \label{eq:boundary-functions}
\begin{aligned}
L(\beta)&\coloneqq \tfrac{3}{16}(1+\beta)^2-\tfrac12, \qquad & \ell(\phi)&\coloneqq \tfrac43 \phi-\tfrac13, \\
U(\beta)&\coloneqq 1-\tfrac38(1-\beta)^2, \qquad & u(\phi)&\coloneqq \tfrac23 \phi+\tfrac13 
\end{aligned}
\end{equation}
of the exact $(\beta,\phi)$-region \cite[Theorem~11]{KoBuKoMoOm21} and the exact $(\phi,\tau)$-region \cite[Theorem~4]{KoBuSt}.
The following theorem is our main result.

\begin{theorem}[Exact $(\beta,\phi,\tau)$-region]\label{thm:main}
A triple $(\beta,\phi,\tau)\in\RR^3$ equals $(\beta(C),\phi(C),\tau(C))$ for some copula $C\in\CC$ if and only if $-1\le \beta\le 1$, $L(\beta)\le \phi\le U(\beta)$ and $\ell(\phi)\le \tau\le u(\phi)$, where $L, U, \ell, u$ are defined in \eqref{eq:boundary-functions}, that is,
\begin{equation}\label{eq:region}
\Om=\bigl\{(\beta,\phi,\tau)\in\RR^3 \mid  -1\le \beta\le 1,\ L(\beta)\le \phi\le U(\beta),\ \ell(\phi)\le \tau\le u(\phi)\bigr\}.
\end{equation}
\end{theorem}

\begin{figure}
    \centering
    \includegraphics[width=7cm]{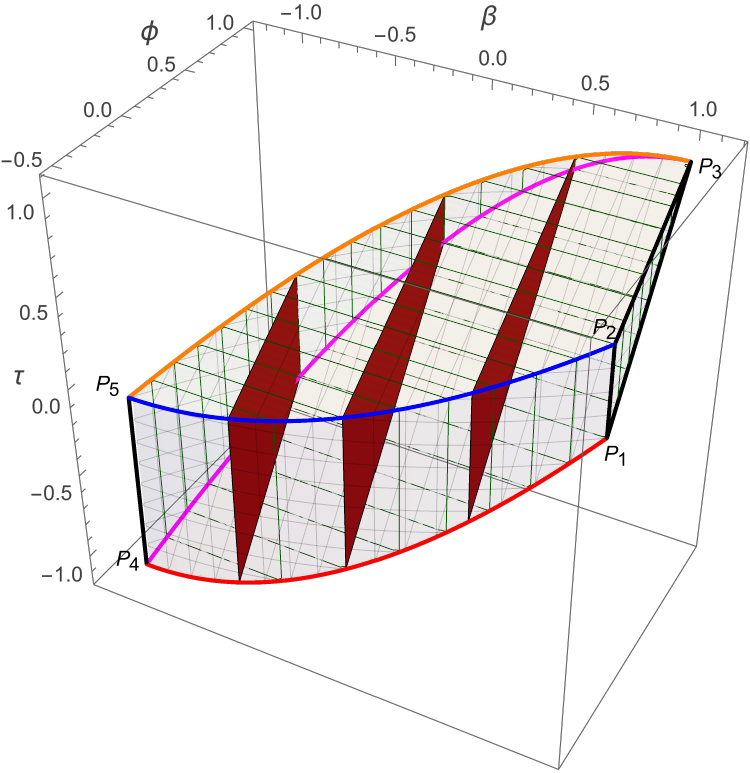} \qquad
    \includegraphics[width=7cm]{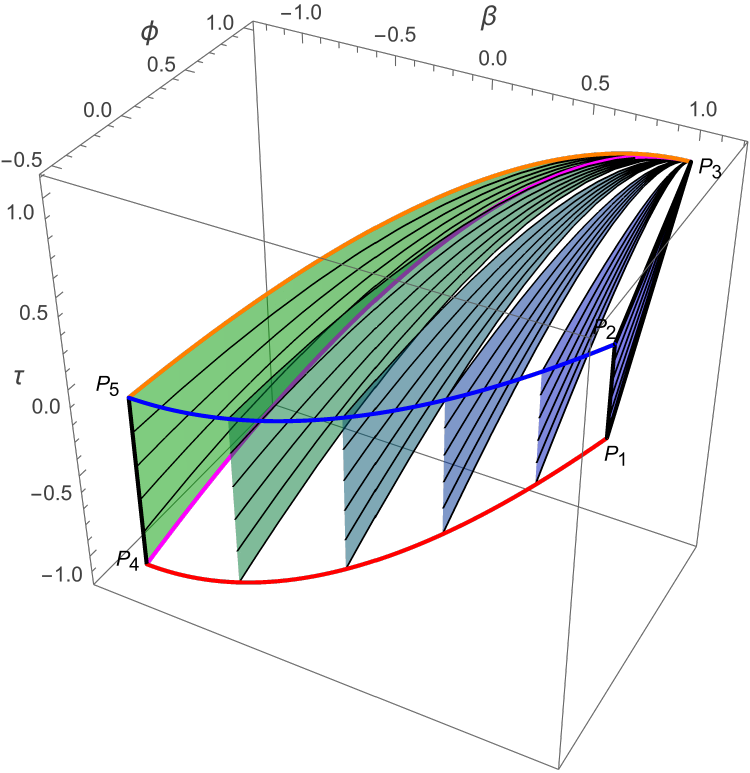}
    \caption{The region $\Om$ determined by $\beta$, $\phi$ and $\tau$ (left) and some curves and surfaces realized by ordinal sums (right), see Remark~\ref{rem:shuffle-realization}.}
    \label{fig:Omega}
\end{figure}

The region $\Om$ is depicted in Figure~\ref{fig:Omega} (left). It is bounded by five surfaces, namely the bottom surface $S_1$, the top surface $S_2$, the front surface $S_3$, the back surface $S_4$ and the triangle $T \coloneqq P_1 P_2 P_3$ (the section of $\Om$ at $\beta=1$), where
\begin{align*} 
   S_1 &\coloneqq \{(\beta, \phi, \tau) \in \RR^3 \mid -1 \leq \beta \leq 1,\ L(\beta)\le \phi\le U(\beta),\ \tau = \ell(\phi) \}, \\
   S_2 &\coloneqq \{(\beta, \phi, \tau) \in \RR^3 \mid -1 \leq \beta \leq 1,\ L(\beta)\le \phi\le U(\beta),\ \tau = u(\phi) \}, \\
   S_3 &\coloneqq \{(\beta, \phi, \tau) \in \RR^3 \mid -1 \leq \beta \leq 1,\ \phi = L(\beta),\ \ell(\phi)\le \tau\le u(\phi) \}, \\
   S_4 &\coloneqq \{(\beta, \phi, \tau) \in \RR^3 \mid -1 \leq \beta \leq 1,\ \phi = U(\beta),\ \ell(\phi)\le \tau\le u(\phi) \}
\end{align*}
and
\[P_1=(1, \tfrac14, 0), \qquad P_2=(1, \tfrac14, \tfrac12), \qquad P_3=(1, 1, 1).\]
The surfaces $S_1, S_2, S_3, S_4$ intersect in the curves
\begin{align*} 
    R_1 &\coloneqq S_1 \cap S_3 =  \{(\beta, \phi, \tau) \in \RR^3 \mid -1 \leq \beta \leq 1, \phi = L(\beta), \tau = \ell(\phi) \}, \\
    R_2 &\coloneqq S_1 \cap S_4 = \{(\beta, \phi, \tau) \in \RR^3 \mid -1 \leq \beta \leq 1, \phi = U(\beta), \tau = \ell(\phi) \}, \\
    R_3 &\coloneqq S_2 \cap S_3 = \{(\beta, \phi, \tau) \in \RR^3 \mid -1 \leq \beta \leq 1, \phi = L(\beta), \tau = u(\phi) \}, \\
    R_4 &\coloneqq S_2 \cap S_4 =  \{(\beta, \phi, \tau) \in \RR^3 \mid -1 \leq \beta \leq 1, \phi = U(\beta), \tau = u(\phi) \}.
\end{align*}
The curves $R_1, R_2, R_3$ and $R_4$ are drawn in red, magenta, blue and orange, respectively, in Figure~\ref{fig:Omega}. The surfaces $S_3$ and $S_4$ intersect in the line segment $P_4 P_5$ (the degenerate section of $\Om$ at $\beta=-1$), where
\[P_4=(-1, -\tfrac12, -1), \qquad P_5=(-1, -\tfrac12, 0),\]
while the surfaces $S_1$ and $S_2$ intersect in a single point $P_3$. The surfaces $S_1, S_2, S_3$ intersect with the triangle $T$ in line segments
\[ S_1 \cap T = P_1 P_3, \qquad  S_2 \cap T = P_2 P_3, \qquad  S_3 \cap T = P_1 P_2, \]
while the surface $S_4$ intersects with the triangle $T$ in the point $P_3$. The intersection of the region $\Om$ with each plane $\beta = \beta_0$ for $\beta_0 \in (-1, 1)$ is a trapezoid with two sides parallel to the $\tau$-axis. The trapezoids for $\beta_0 = -\tfrac12, \beta_0 = 0$ and $\beta_0 = \tfrac12$ are depicted in dark red in Figure~\ref{fig:Omega} (left).

Note that the extreme points of the region $\Om$ are exactly the points on the curves $R_1, R_2, R_3, R_4$: since $L$ is strictly convex, $U$ is strictly concave and $\ell$, $u$ are affine, a point $(\beta,\phi,\tau)\in\Om$ is extreme if and only if $\phi\in\{L(\beta),U(\beta)\}$ and $\tau\in\{\ell(\phi),u(\phi)\}$. Note also that the projection of the region $\Om$ onto the $(\beta,\phi)$-plane is the exact region determined by $\beta$ and $\phi$, 
\[ \Omega_{\beta, \phi} = \{(\beta, \phi) \in \RR^2 \mid -1 \leq \beta \leq 1,\ L(\beta)\le \phi\le U(\beta) \} ,\] 
the projection of the region $\Om$ onto the $(\beta,\tau)$-plane is the exact region determined by $\beta$ and $\tau$, 
\[ \Omega_{\beta, \tau} = \{(\beta, \tau) \in \RR^2 \mid -1 \leq \beta \leq 1,\ \tfrac14 (1+\beta)^2 - 1 \leq \tau \leq 1 - \tfrac14(1-\beta)^2 \}, \] 
(this is less immediate, see Corollary~\ref{cor:beta-tau}), and the projection of the region $\Om$ onto the $(\phi,\tau)$-plane is the exact region determined by $\phi$ and $\tau$, \[ \Omega_{\phi, \tau} = \{(\phi, \tau) \in \RR^2 \mid -\tfrac12 \leq \phi \leq 1,\ \ell(\phi)\le \tau\le u(\phi) \}, \]  
which is the triangle with vertices $P'_4=(-\tfrac12,-1)$, $P'_5=(-\tfrac12,0)$ and $P'_3=(1,1)$ (see \cite{KoBuSt}).

The inclusion ``$\subseteq$'' in \eqref{eq:region} follows directly from the known bivariate regions.
Theorem~\ref{thm:main} shows that there is no further constraint: $\Om$ is exactly the intersection of the two cylinders over the $(\beta,\phi)$- and $(\phi,\tau)$-regions.
Equivalently, for every admissible pair $(\beta, \phi)$ the fiber $\{\tau\mid (\beta,\phi,\tau)\in\Om\}$ is the full interval $[\ell(\phi),u(\phi)]$, so that, given the value of Spearman's footrule, the value of Blomqvist's beta imposes no additional restriction on Kendall's tau.
This purely bivariate constraint structure contrasts with the regions determined in \cite{KoBuMo26,KoBuLaMoSt}, whose boundaries contain faces described by genuinely ternary constraints.

The proof provides an explicit construction.
We first give two $1$-parameter families of shuffles of $M$ attaining all points along the curves $R_1$ and $R_3$, see Propositions~\ref{prop:lower-seed} and~\ref{prop:upper-seed}.
We then construct two $2$-parameter families of shuffles of $M$ that together attain all points on the surface $S_3$, see Proposition~\ref{prop:S_3}.
Finally, ordinal sums with $M$ transform these families into $3$-parameter families, attaining every point in $\Om$.
The key observation here is how the ordinal sums transform the coefficients $\beta$, $\phi$ and $\tau$, see \eqref{eq:ordinal-sum} below.

The rest of the paper is organized as follows.
Section~\ref{sec:prelim} collects notation, the two known bivariate regions and the transformation behavior of the three coefficients under ordinal sums.
In Section~\ref{sec:seeds} we construct two $1$-parameter families of shuffles of $M$, realizing the curves $R_1$ and $R_3$, and two $2$-parameter families of shuffles of $M$, realizing the surface $S_3$. 
Section~\ref{sec:proof} contains the proof of Theorem~\ref{thm:main}, explicit copulas realizing the surface $S_4$ and the triangle $T$ (Examples~\ref{ex:S4} and~\ref{ex:T}), the geometry of the fixed-footrule sections of $\Om$ (Corollary~\ref{cor:fixed-phi}) and the $(\beta,\tau)$-projection of $\Om$ (Corollary~\ref{cor:beta-tau}), which recovers the $(\beta,\tau)$-region given in \cite[Section~8]{Krus} (see also \cite[Exercise~5.17]{Nels} and \cite[Corollary~6]{KoBuKoMoOm19}). 
Section~\ref{sec:conclusion} closes with the volume of $\Om$ and some open problems.

\section{Preliminaries}\label{sec:prelim}

Throughout, an integral with respect to a copula denotes the Lebesgue--Stieltjes integral with respect to the probability measure induced by that copula.
Our proof uses the following two known exact regions.

\begin{proposition}[Known bivariate regions]\label{prop:known}
Every copula $C\in\CC$ satisfies
\begin{align}
\ell(\phi(C))&\le\tau(C)\le u(\phi(C)),\label{eq:ft-region}\\
L(\beta(C))&\le\phi(C)\le U(\beta(C)),\label{eq:bf-region}
\end{align}
with $L,U,\ell,u$ defined in \eqref{eq:boundary-functions}, and for every $\phi\in[-\frac12,1]$ (resp.\ $\beta\in[-1,1]$) both bounds in \eqref{eq:ft-region} (resp.\ \eqref{eq:bf-region}) are attained.
\end{proposition}

The $(\phi,\tau)$-region \eqref{eq:ft-region} was determined in \cite[Theorem~4]{KoBuSt}, while the $(\beta,\phi)$-region \eqref{eq:bf-region} was given in \cite[Theorem~11]{KoBuKoMoOm21}.

A \emph{shuffle of $M$} in the sense of \cite{Miku} is a copula $C=M(n,J,\pi,\varepsilon)$ specified by a partition $J=(x_1,\dots,x_{n-1})$, $0=x_0\le x_1\le\dots\le x_n=1$, of $\II$, which splits the unit square $\II^2$ into $n$ (vertical) strips of widths $s_i\coloneqq x_i-x_{i-1}$, a permutation $\pi$ of $\{1,\dots,n\}$ (written below in one-line notation) and a vector of signs $\varepsilon\in\{-1,1\}^n$.
We write $0=y_0\le y_1\le\dots\le y_n=1$ for the partition of the second coordinate, which splits the unit square $\II^2$ into (horizontal) bands of heights $y_k-y_{k-1}=s_{\pi^{-1}(k)}$, i.e., $y_k\coloneqq\sum_{j=1}^{k}s_{\pi^{-1}(j)}$. The copula $M(n,J,\pi,\varepsilon)$ distributes mass $s_i$ uniformly on the main diagonal (if $\varepsilon_i=1$) or the opposite diagonal (if $\varepsilon_i=-1$) of the square $[x_{i-1},x_i]\times[y_{\pi(i)-1},y_{\pi(i)}]$, for $i=1,\dots,n$.
We allow zero strip widths in this definition, so that the endpoint cases used below are included as possibly degenerate shuffles of $M$; equivalently, they may be obtained by continuity from the corresponding nondegenerate shuffles.

Let $\{(a_k, b_k) \mid k=1,2, \ldots, n\}$ be a finite family of disjoint open subintervals of $\II$ and $\{C_k \mid k=1,2, \ldots, n\}$ a family of copulas. Then the \emph{ordinal sum} (sometimes also called \emph{$M$-ordinal sum}) of $\{C_k \mid k=1,2, \ldots, n\}$ with respect to $\{(a_k, b_k) \mid k=1,2, \ldots, n\}$ is a copula $B$ defined by 
\begin{equation}
B(u, v) = \begin{cases}
		a_k + (b_k-a_k)C_k\big(\frac{u-a_k}{b_k-a_k}, \frac{v-a_k}{b_k-a_k}\big); & (u, v) \in (a_k, b_k)^2, k=1, 2, \dots, n,\\
			\min\{u,v\}; & \text{otherwise},
    \end{cases}
\end{equation}
(see \cite{DuKlSaSe}, compare also \cite[Section~3.2.2]{Nels}).

For $a\in[0,\frac12]$ and $C\in\CC$, let $B_{a,C}$ denote the ordinal sum of $\{C\}$ with respect to the single interval $(a,1-a)$, i.e., the copula that rescales $C$ to the central square $[a,1-a]^2$ and coincides with $M$ elsewhere on $\II^2$. For $a=\frac12$ we set $B_{1/2,C}\coloneqq M$.
Denoting the side length of the square $[a,1-a]^2$ by $\alpha\coloneqq 1-2a\in[0,1]$, the formulas of \cite[Lemma~2]{KoBuLaMoSt} for $\tau(B_{a,C})$, $\phi(B_{a,C})$ and $\beta(B_{a,C})$ read as
\begin{equation}\label{eq:ordinal-sum}
\begin{aligned}
\tau(B_{a,C})&=\alpha^2\tau(C)+1-\alpha^2,\\
\phi(B_{a,C})&=\alpha^2\phi(C)+1-\alpha^2,\\
\beta(B_{a,C})&=\alpha\,\beta(C)+1-\alpha.
\end{aligned}
\end{equation}
Here, $\tau$ and $\phi$ are transformed by the same affine map $z\mapsto\alpha^2z+1-\alpha^2$, whereas $\beta$ is transformed by $z\mapsto\alpha z+1-\alpha$.

\section{Realizing the front surface \texorpdfstring{$S_3$}{S3} of the region \texorpdfstring{$\Om$}{Omega}}\label{sec:seeds}

To prove Theorem~\ref{thm:main}, we first need copulas that attain the extreme values $\ell(\phi)$ and $u(\phi)$ of Kendall's tau for a prescribed pair $(\beta,\phi)$ on the lower boundary $\phi=L(\beta)$ of the $(\beta,\phi)$-region, i.e., copulas attaining the curves $R_1$ and $R_3$ of the region $\Om$.
We construct them as shuffles of $M$, shown in Figures~\ref{fig:S1} and~\ref{fig:S2}, and compute their Kendall's tau using the following formula.

\begin{lemma}[Kendall's tau of a shuffle of $M$]\label{lem:shuffle-tau}
Every shuffle of $M$, $C=M(n,J,\pi,\varepsilon)$, with strip widths $s_1,\dots,s_n$ satisfies
\begin{equation}\label{eq:shuffle-tau}
\tau(C)=\sum_{i=1}^n\varepsilon_i s_i^2
+2\sum_{1\le i<j\le n}\sgn\bigl(\pi(j)-\pi(i)\bigr)\,s_i s_j.
\end{equation}
\end{lemma}

\begin{proof}
Let $(X,Y)$ and $(X',Y')$ be independent random vectors with common distribution function $C$.
Deleting zero-width strips and relabeling the rest does not change either the copula or the right-hand side of \eqref{eq:shuffle-tau}, so we may assume that $s_i>0$ for every $i$.
By \cite[Section~5.1.1]{Nels},
\[
\tau(C)=\mathbb{P}\bigl((X-X')(Y-Y')>0\bigr)-\mathbb{P}\bigl((X-X')(Y-Y')<0\bigr).
\]
The probability of both points $(X,Y)$ and $(X',Y')$ falling into the same strip $i$ is $s_i^2$.
Given this event, they are almost surely concordant if $\varepsilon_i=1$ and almost surely discordant if $\varepsilon_i=-1$, giving the contribution $\varepsilon_i s_i^2$.
The two points fall into distinct strips $i<j$ with probability $2s_is_j$.
In this case they are concordant exactly when the vertical order of the bands agrees with the horizontal order of the strips, that is, when $\pi(i)<\pi(j)$, giving the contribution $2\sgn(\pi(j)-\pi(i))s_is_j$.
These events are exhaustive up to null sets, and summing the contributions gives \eqref{eq:shuffle-tau}.
\end{proof}

\begin{proposition}[Realizing the curve $R_1$]\label{prop:lower-seed}
For $r\in[0,\frac12]$ let
\[
A_r\coloneqq M\bigl(4,\bigl(\tfrac12-r,\tfrac12,\tfrac12+r\bigr),(4,2,3,1),(-1,-1,-1,-1)\bigr).
\]
Then
\[
\bigl(\beta(A_r),\phi(A_r),\tau(A_r)\bigr)
=\bigl(4r-1,\;L(4r-1),\;\ell(L(4r-1))\bigr)
=\bigl(4r-1,\;3r^2-\tfrac12,\;4r^2-1\bigr).
\]
\end{proposition}

\begin{proof}
Write $d\coloneqq \frac12-r$, so that the strip widths of $A_r$ are $(d,r,r,d)$ and the band heights are likewise $(d,r,r,d)$.
The permutation $(4,2,3,1)$ maps the first strip to the top band $[\frac12+r,1]$ and the second strip, of width $r$, to the band $[d,\frac12]\subseteq[0,\frac12]$, so the mass in $[0,\frac12]^2$ equals $r$ and $\beta(A_r)=4r-1$ by \eqref{eq:beta-def}.
A direct computation of the diagonal section gives
\[
\delta_{A_r}(u)=
\begin{cases}
0; & 0\le u\le \frac12-\frac r2,\\
2u+r-1; & \frac12-\frac r2<u\le \frac12,\\
r; & \frac12<u\le \frac12+\frac r2,\\
2u-1; & \frac12+\frac r2<u\le 1,
\end{cases}
\]
so $\int_0^1\delta_{A_r}(u)\de u=\frac14+\frac{r^2}{2}$ and, by \eqref{eq:phi-def}, $\phi(A_r)=3r^2-\frac12=L(4r-1)$.
Using Lemma~\ref{lem:shuffle-tau} for $C=A_r$, the first sum in equation~\eqref{eq:shuffle-tau} equals $-2d^2-2r^2$ and the second term equals $2(r^2-d^2-4dr)$, so
\[
\tau(A_r)=-4d^2-8dr=-4d(d+2r)=-4\bigl(\tfrac12-r\bigr)\bigl(\tfrac12+r\bigr)=4r^2-1,
\]
and indeed $\ell(3r^2-\frac12)=4r^2-1$.
\end{proof}

The mass distributions of copula $A_r$ and its ordinal sum $B_{a,A_r}$ (needed later) in the case of $a=\tfrac{3}{20}$ and $r=\tfrac{5}{14}$ are shown in Figure~\ref{fig:S1}. 
The family $A_r$ traces the curve $R_1$ of $\Om$ with $(\beta,\phi,\tau)=(4r-1,3r^2-\frac12,4r^2-1)$.
The copula spreads its mass uniformly on segments parallel to the opposite diagonal.
As $r$ increases, the two central segments lengthen while the two outer segments shrink, degenerating at $r=\frac12$.

\begin{figure}[ht]
\centering
\includegraphics{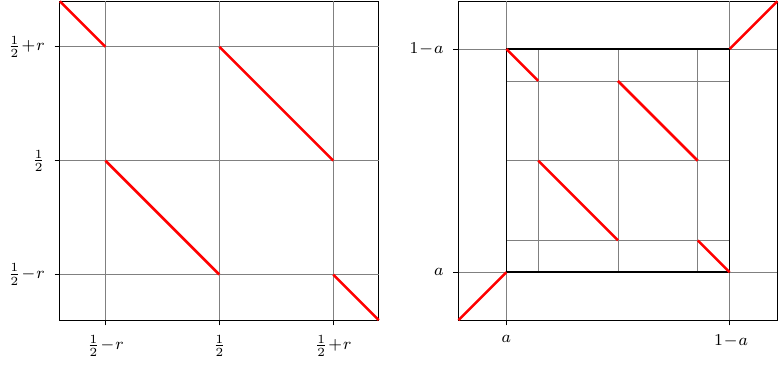}
    \caption{The mass distribution of copula $A_r$ from Proposition \ref{prop:lower-seed} (left) and copula $B_{a,A_r}$ (right) in the case of $a=\tfrac{3}{20}$ and $r=\tfrac{5}{14}$.}
    \label{fig:S1}
\end{figure}

\begin{proposition}[Realizing the curve $R_3$]\label{prop:upper-seed}
For $q\in[0,\frac14]$ let
\[
D_q\coloneqq M\bigl(6,\bigl(q,\tfrac12-q,\tfrac12,\tfrac12+q,1-q\bigr),(3,5,1,6,2,4),(1,1,1,1,1,1)\bigr).
\]
Then
\[
\bigl(\beta(D_q),\phi(D_q),\tau(D_q)\bigr)
=\bigl(8q-1,\;L(8q-1),\;u(L(8q-1))\bigr)
=\bigl(8q-1,\;12q^2-\tfrac12,\;8q^2\bigr).
\]
\end{proposition}

\begin{proof}
The strip widths of $D_q$ are $(q,d,q,q,d,q)$ with $d\coloneqq \frac12-2q$.
The permutation $\pi$ of $D_q$ maps the three strip intervals contained in $[0,\frac12]$ into the band intervals $[\frac12-q,\frac12]$, $[\frac12+q,1-q]$ and $[0,q]$, respectively, so the mass in $[0,\frac12]^2$ equals $2q$ and $\beta(D_q)=8q-1$ by \eqref{eq:beta-def}.

We next compute Spearman's footrule from the diagonal section of $D_q$. Observe that
\[ \delta_{D_q}(u) =  \begin{cases}
            0; &  0 \le u \le \tfrac12 - q, \\
            2u-1+2q; &  \tfrac12 - q < u \le \tfrac12, \\
            2q; &  \tfrac12 < u \le \tfrac12 + q, \\
            2u-1; &  \tfrac12 + q < u \le 1.
            \end{cases}
\]
Hence, by a short calculation, we obtain
\[
\int_0^1\delta_{D_q}(u)\de u= \tfrac14+2q^2
\qquad\text{and}\qquad
\phi(D_q)=6\Bigl(\tfrac14+2q^2\Bigr)-2=12q^2-\tfrac12.
\]

Finally, Kendall's tau of $D_q$ follows from Lemma~\ref{lem:shuffle-tau}: the first sum in equation~\eqref{eq:shuffle-tau} with $C=D_q$ equals $4q^2+2d^2$, while the second term equals $2(2q^2-d^2)$, so
\[
\tau(D_q)=4q^2+2d^2+2(2q^2-d^2)=8q^2.
\]
Thus $L(8q-1)=12q^2-\frac12=\phi(D_q)$ and $u(12q^2-\frac12)=8q^2=\tau(D_q)$.
\end{proof}

The mass distributions of copula $D_q$ and its ordinal sum $B_{a,D_q}$ (needed later) in the case of $a=\tfrac{3}{20}$ and $q=\tfrac{5}{28}$ are shown in Figure~\ref{fig:S2}. 
The family $D_q$ traces the curve $R_3$ of $\Om$ with $(\beta,\phi,\tau)=(8q-1, 12q^2-\tfrac12,8q^2)$.
The copula spreads its mass uniformly on segments parallel to the main diagonal.
As $q$ increases, the four strips of width $q$ grow while the two strips of width $d=\frac12-2q$ shrink, degenerating at $q=\frac14$.

\begin{figure}[ht]
\centering
\includegraphics{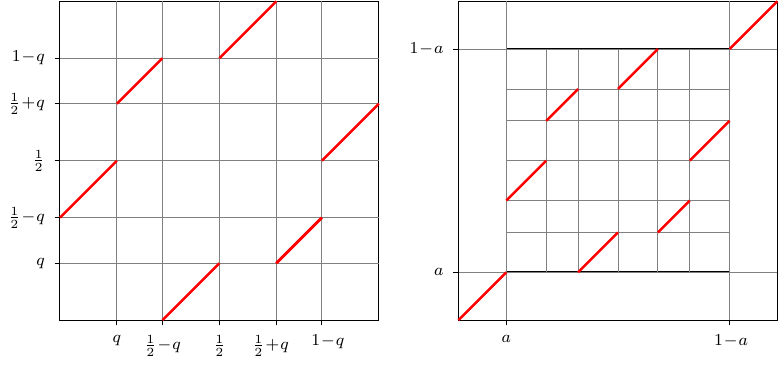}
    \caption{The mass distribution of copula $D_q$ from Proposition \ref{prop:upper-seed} (left) and copula $B_{a,D_q}$ (right) in the case of $a=\tfrac{3}{20}$ and $q=\tfrac{5}{28}$.}
    \label{fig:S2}
\end{figure}

\begin{remark}\label{rem:boundary-edges}
The families $A_r$ and $D_q$ from Propositions~\ref{prop:lower-seed} and \ref{prop:upper-seed} give explicit copulas tracing the curves $R_1$ and $R_3$ in Figure~\ref{fig:Omega}.
The curves $R_2$ and $R_4$ can be realized with ordinal sums as follows.
Let $\alpha\in[0,1]$, $a=(1-\alpha)/2$, and $\beta=1-2\alpha$.
Applying \eqref{eq:ordinal-sum} first to $A_0=W$ and then to $D_0$ gives
\begin{align*}
(\beta(B_{a,A_0}),\phi(B_{a,A_0}),\tau(B_{a,A_0}))
&=\bigl(1-2\alpha,\,1-\tfrac32\alpha^2,\,1-2\alpha^2\bigr)=\bigl(\beta,U(\beta),\ell(U(\beta))\bigr),\\
(\beta(B_{a,D_0}),\phi(B_{a,D_0}),\tau(B_{a,D_0}))
&=\bigl(1-2\alpha,\,1-\tfrac32\alpha^2,\,1-\alpha^2\bigr)=\bigl(\beta,U(\beta),u(U(\beta))\bigr).
\end{align*}
As $\alpha$ runs over the interval $[0,1]$, the first point traces the curve $R_2$, while the second point traces the curve $R_4$. 
\end{remark}

In our next proposition, we realize all points on the surface $S_3$ of the region $\Om$ with shuffles of $M$.

\begin{proposition}[Realizing the front surface $S_3$]\label{prop:S_3}
For any $b \in [0, \tfrac12]$ and $c \in [0, \tfrac12-b]$ let $U_{b,c}$ be a shuffle of $M$
\begin{align*}
   U_{b,c}\coloneqq M\big(8, &(\tfrac12 - b - c, \tfrac12 - \tfrac{b}{2} - c, \tfrac12 - \tfrac{b}{2}, \tfrac12, \tfrac12 + \tfrac{b}{2}, \tfrac12 + \tfrac{b}{2} + c, \tfrac12 + b + c), \\
   &(8,4,6,2,7,3,5,1), (-1, -1, 1, -1, -1, 1, -1, -1)\big). 
\end{align*}
Furthermore, for any $b \in [0, \tfrac12]$ and $c \in [0, \tfrac{b}{2}]$ let $V_{b,c}$ be a shuffle of $M$
\begin{align*}
   V_{b,c}\coloneqq M\big(10, &(c, \tfrac{b}{2}, \tfrac12 - \tfrac{b}{2}, \tfrac12  - \tfrac{b}{2} + c, \tfrac12, \tfrac12 + \tfrac{b}{2} - c, \tfrac12 + \tfrac{b}{2}, 1 - \tfrac{b}{2}, 1 - c), \\
   &(4,5,8,1,2,9,10,3,6,7), (1, -1, 1, 1, -1, -1, 1, 1, -1, 1)\big). 
\end{align*}
Then $U_{b,0}=A_b$, $V_{b,b/2}=D_{b/2}$ (see Propositions~\ref{prop:lower-seed} and \ref{prop:upper-seed}), $U_{b,1/2-b} = V_{b,0}$, and
\[\begin{aligned}
S_3 &=\{(\beta(U_{b,c}), \phi(U_{b,c}), \tau(U_{b,c})) \mid b \in [0, \tfrac12],\ c \in [0, \tfrac12-b]\}\\
& \qquad \cup \{(\beta(V_{b,c}), \phi(V_{b,c}), \tau(V_{b,c})) \mid b \in [0, \tfrac12],\ c \in [0, \tfrac{b}{2}]\}.
\end{aligned}\]
\end{proposition}

\begin{proof}
For any $b \in [0, \tfrac12]$ we have 
\begin{align*}
    U_{b,0} &= M\big(6, (\tfrac12 - b, \tfrac12-\tfrac{b}{2},\tfrac12, \tfrac12+\tfrac{b}{2}, \tfrac12 + b), (6,3,2,5,4,1), (-1, -1, -1, -1, -1, -1)\big)\\
    &=M\big(4, (\tfrac12 - b, \tfrac12, \tfrac12 + b), (4,2,3,1), (-1, -1, -1, -1)\big) = A_b,
\end{align*}
where $A_b$ is the copula from Proposition~\ref{prop:lower-seed} with $r=b$. The point $(\beta(U_{b,0}), \phi(U_{b,0}), \tau(U_{b,0}))$ lies on the curve $R_1$ of the region $\Om$, see Proposition~\ref{prop:lower-seed}.
Furthermore, 
\[V_{b,b/2} = M\big(6, (\tfrac{b}{2}, \tfrac12 - \tfrac{b}{2}, \tfrac12, \tfrac12 + \tfrac{b}{2}, 1 - \tfrac{b}{2}), (3,5,1,6,2,4), (1, 1, 1, 1, 1, 1)\big) = D_{b/2},\]
where $D_{b/2}$ is the copula from Proposition~\ref{prop:upper-seed} with $q=\tfrac{b}{2}$. The corresponding point $(\beta(V_{b,b/2}), \phi(V_{b,b/2}), \tau(V_{b,b/2}))$ lies on the curve $R_3$ of the region $\Om$, see Proposition~\ref{prop:upper-seed}.
In addition, 
\[U_{b,1/2-b} = M\big(6, (\tfrac{b}{2}, \tfrac12 - \tfrac{b}{2}, \tfrac12, \tfrac12 + \tfrac{b}{2}, 1 - \tfrac{b}{2}), 
(3,5,1,6,2,4), (-1, 1, -1, -1, 1, -1)\big)= V_{b,0}.\]

Since $U_{b,c}(\tfrac12,\tfrac12) = V_{b,c}(\tfrac12,\tfrac12) = b$ by the diagonal sections below, we have $\beta(U_{b,c}) =  \beta(V_{b,c}) =  4b - 1$. 
The diagonal sections of the two copulas are equal to
\[ \delta_{U_{b,c}}(u) =  \delta_{V_{b,c}}(u) = \begin{cases}
            0; &  0 \le u \le \tfrac12 - \tfrac{b}{2}, \\
            2u+b-1; &  \tfrac12 - \tfrac{b}{2} < u \le \tfrac12, \\
            b; &  \tfrac12 < u \le \tfrac12 + \tfrac{b}{2}, \\
            2u-1; &  \tfrac12 + \tfrac{b}{2} < u \le 1,
            \end{cases}
\]
for any $b \in [0, \tfrac12]$ regardless of the value of $c$, so we obtain
\begin{equation}\label{eq:phiUV}
    \phi(U_{b,c}) = \phi(V_{b,c}) = 3b^2 - \tfrac12.
\end{equation} 
Note that 
\[L(\beta(U_{b,c})) = L(\beta(V_{b,c})) = 3b^2 - \tfrac12 = \phi(U_{b,c}) = \phi(V_{b,c}),\]
so the points $(\beta(U_{b,c}), \phi(U_{b,c}), \tau(U_{b,c}))$ and $(\beta(V_{b,c}), \phi(V_{b,c}), \tau(V_{b,c}))$ lie on the surface $S_3$ for any $b\in [0, \tfrac12]$ and appropriate $c$ (recall that Kendall's tau satisfies \eqref{eq:ft-region} by Proposition~\ref{prop:known}).

Using equation \eqref{eq:shuffle-tau}, we obtain $\tau(V_{b,0}) = 0$.
Thus, the points $(\beta(V_{b,0}), \phi(V_{b,0}), \tau(V_{b,0}))$ lie on the intersection of the surface $S_3$ and the coordinate plane $\tau = 0$. 

For any fixed $b \in [0, \tfrac12]$ the points $(\beta(U_{b,c}), \phi(U_{b,c}), \tau(U_{b,c}))$ (for $c \in [0, \tfrac12-b]$) and the points $(\beta(V_{b,c}), \phi(V_{b,c}), \tau(V_{b,c}))$ (for $c \in [0, \tfrac{b}{2}]$) lie on the vertical line segment
\[
\beta = 4b-1,\ \phi = 3b^2 - \tfrac12,\ \tau \in [\tau(U_{b,0}),\tau(V_{b,b/2})]=[\tau(A_{b}),\tau(D_{b/2})]=[4b^2-1, 2b^2].
\]
Now, for any point $P=(\beta, \phi, \tau)$ from the surface $S_3$ let $b = \tfrac{\beta+1}{4}$, so that $\beta(U_{b,c}) = \beta(V_{b,c}) = \beta$ and $\phi(U_{b,c}) = \phi(V_{b,c}) = \phi$. If $\tau \le 0$ then there exists $c \in [0, \tfrac12-b]$ such that $\tau(U_{b,c}) = \tau$, since $\tau(U_{b,c})$ is a continuous function in $c$. Similarly, if $\tau \ge 0$ then there exists $c \in [0, \tfrac{b}{2}]$ such that $\tau(V_{b,c}) = \tau$. So, either $(\beta(U_{b,c}), \phi(U_{b,c}), \tau(U_{b,c})) = P$ or $(\beta(V_{b,c}), \phi(V_{b,c}), \tau(V_{b,c})) = P$ for some $b, c$.
\end{proof}

Figure~\ref{fig:S3} shows the mass distribution of copulas $U_{b,c}$, $V_{b,c}$, and $V_{b,0}=U_{b,1/2-b}$ for $b=\tfrac{3}{10}$ and $c=\tfrac{1}{10}$.

\begin{figure}[ht]
\centering
\includegraphics{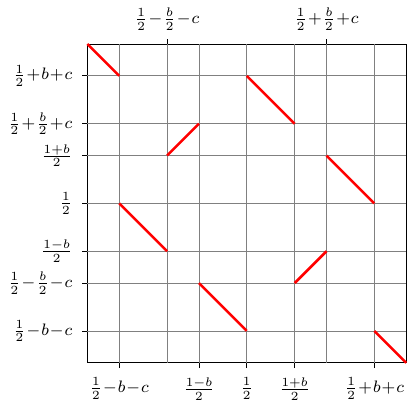}
\includegraphics{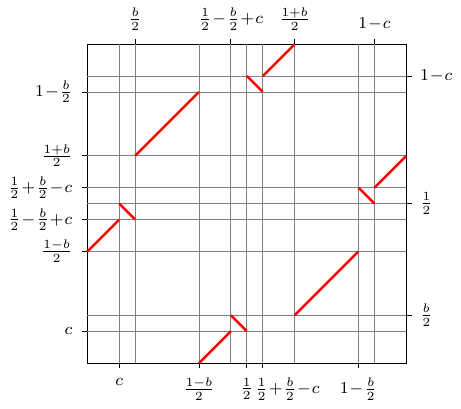} \\[5mm]
\includegraphics{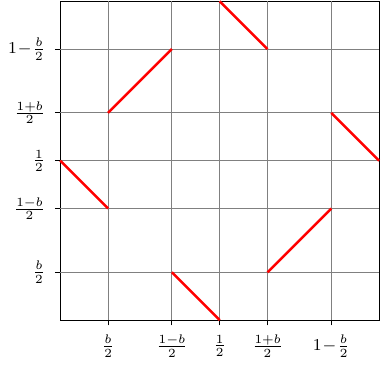}
    \caption{The mass distribution of copulas $U_{b,c}$ (top left), $V_{b,c}$ (top right), and $V_{b,0}$ (bottom) for $b=\tfrac{3}{10}$ and $c=\tfrac{1}{10}$ from Proposition~\ref{prop:S_3}.}
    \label{fig:S3}
\end{figure}

\section{Proof of the main result and some corollaries}\label{sec:proof}

We proceed by proving Theorem~\ref{thm:main}, the main result of our paper.

\begin{proof}[Proof of Theorem~\ref{thm:main}]
If $(\beta,\phi,\tau)=(\beta(C),\phi(C),\tau(C))$ for some $C\in\CC$, then $\beta\in[-1,1]$, and the constraints $L(\beta)\le \phi\le U(\beta)$ and $\ell(\phi)\le \tau\le u(\phi)$ hold by Proposition~\ref{prop:known}.

Conversely, fix $\beta\in[-1,1]$, $\phi\in[L(\beta),U(\beta)]$ and $\tau\in[\ell(\phi),u(\phi)]$.
The idea is to realize $(\beta,\phi,\tau)$ by an ordinal sum $B_{a,E}$, where $E$ realizes a point on the surface $S_3$ and $\alpha=1-2a$ is the side length of the central square of the ordinal sum.
By \eqref{eq:ordinal-sum}, the requirement $\beta(B_{a,E})=\beta$ determines $\beta(E)=(\beta-1+\alpha)/\alpha$, and then $\phi(E)=L(\beta(E))$ yields
\[\phi(B_{a,E})=\alpha^2L\Bigl(\frac{\beta-1+\alpha}{\alpha}\Bigr)+1-\alpha^2=\tfrac{3}{16}(2\alpha+\beta-1)^2-\tfrac32\alpha^2+1,\]
which depends only on the square size $\alpha$. We therefore need the auxiliary function $F_\beta\colon [\tfrac{1-\beta}{2},1] \to \RR$ defined by 
\[F_\beta(\alpha) \coloneqq \tfrac{3}{16} (2\alpha+\beta-1)^2 -\tfrac32 \alpha^2 +1.\]
Observe that $F_\beta(1)= L(\beta)$ and $F_\beta(\tfrac{1-\beta}{2})=U(\beta)$.
Furthermore, 
\[F'_\beta(\alpha)= \tfrac34 (2\alpha+\beta-1) -3\alpha = -\tfrac32 \alpha -\tfrac34 (1-\beta) <0\]
as soon as $(\alpha,\beta) \neq (0,1)$,
so the function $F_\beta$ maps the interval $[\tfrac{1-\beta}{2},1]$ bijectively onto the interval $[L(\beta),U(\beta)]$. Hence, the equation $F_\beta(\alpha)=\phi$ has a unique solution
\[\alpha=\sqrt{\tfrac12 (1-\beta)^2+\tfrac43(1-\phi)}-\tfrac12(1-\beta) \in [\tfrac{1-\beta}{2},1].\] 
If $\beta=\phi=1$, then $\ell(\phi)=u(\phi)=1$, hence $\tau=1$. In this case the point $(\beta,\phi,\tau)=(1,1,1)$ is attained by $M$.

Assume, henceforth, that $(\beta,\phi) \neq (1,1).$ If $\beta <1$ then $\alpha \geq \tfrac{1-\beta}{2}>0$. If $\beta=1$ and $\phi <1= U(1)$ then $\alpha > \tfrac{1-\beta}{2} = 0$ by bijectivity of $F_\beta$.
Hence, $\alpha \in (0,1]$, which allows us to define
\begin{equation} \label{eq:primes}
\begin{aligned}
\beta' &\coloneqq \frac{\beta-1+\alpha}{\alpha},\\
\phi' &\coloneqq \frac{\phi-1+\alpha^2}{\alpha^2}=\frac{F_\beta(\alpha)-1+\alpha^2}{\alpha^2}=\tfrac{3}{16}\Big(1+\frac{\beta-1+\alpha}{\alpha}\Big)^2-\tfrac12=L(\beta'), \quad\text{and}\\
\tau' &\coloneqq \frac{\tau-1+\alpha^2}{\alpha^2}.
\end{aligned}    
\end{equation}
Furthermore, let $a \coloneqq \frac{1-\alpha}{2} \in [0,\tfrac12)$.
Since $\alpha \in [\tfrac{1-\beta}{2},1]$, we have $\beta' \in [-1,\beta] \subseteq [-1,1]$, and since $\tau \in [\ell(\phi),u(\phi)]=[\tfrac43\phi-\tfrac13,\tfrac23\phi+\tfrac13]$, we have $\tau' \in [\ell(\phi'),u(\phi')]$, because $\ell$ and $u$ are affine with $\ell(1)=u(1)=1$ and hence commute with the map $z\mapsto(z-1+\alpha^2)/\alpha^2$. 
This implies that $(\beta',\phi',\tau')$ lies on the surface $S_3$. By Proposition~\ref{prop:S_3} there exists $E\in\{U_{b,c},V_{b,c}\}$, for suitable $b$ and $c$ with
\[
\bigl(\beta(E),\phi(E),\tau(E)\bigr)=\bigl(\beta',\phi',\tau'\bigr).
\]
By equation \eqref{eq:phiUV} we have
\begin{equation}\label{eq:b}
    b = \sqrt{\tfrac13(\phi'+\tfrac12)} = \frac{2\alpha +\beta -1}{4\alpha} \in [0,\tfrac{\beta+1}{4}] \subseteq [0,\tfrac12].
\end{equation}
Now, let $C=B_{a,E}$. Then, by equation~\eqref{eq:ordinal-sum},
\begin{align*}
\beta(C) &= \alpha \beta(E)+ 1-\alpha = \beta, \\
\phi(C) &= \alpha^2 \phi(E)+ 1-\alpha^2 = \phi, \\
\tau(C) &= \alpha^2 \tau(E)+ 1-\alpha^2 = \tau, 
\end{align*}
which finishes the proof.
\end{proof}

\begin{remark} \label{rem:shuffle-realization}
Proposition~\ref{prop:S_3} together with the ordinal sums from the proof of Theorem~\ref{thm:main} shows that every point of $\Om$ is attained by a shuffle of $M$, because $B_{a,C}$ is a shuffle of $M$ whenever $C$ is one.
Since $U_{b,c}$ and $V_{b,c}$ have at most $10$ strips, every point in $\Om$ is attained by a shuffle of $M$ with at most $12$ strips.
Moreover, the shuffles constructed here are exchangeable, i.e., $C(u,v)=C(v,u)$, and radially symmetric, i.e., $C(u,v)=u+v-1+C(1-u,1-v)$, for all $(u,v)\in\II^2$.
Indeed, for each of $A_r$, $D_q$, $U_{b,c}$ and $V_{b,c}$ the permutation $\pi$ is an involution with $s_{\pi(i)}=s_i$ and $\varepsilon_{\pi(i)}=\varepsilon_i$, which makes the copula exchangeable, and $s_{n+1-i}=s_i$, $\varepsilon_{n+1-i}=\varepsilon_i$ and $\pi(n+1-i)=n+1-\pi(i)$ for all $i$, which makes it radially symmetric.
Both properties are preserved by the ordinal sums $B_{a,\cdot}$, and $M$ has both.
Hence, $\Om$ is also the exact region determined by $\beta$, $\phi$ and $\tau$ over the class of exchangeable, radially symmetric copulas.

Figure~\ref{fig:Omega} (right) shows the curves realized by $\{B_{a,U_{b,c}} \mid a \in [0,\tfrac12]\}$ and $\{B_{a,V_{b,c}} \mid a \in [0,\tfrac12]\}$ for some fixed values of $b$ and $c$, and surfaces realized by 
\[\{B_{a,U_{b,c}} \mid a \in [0,\tfrac12], c \in [0, \tfrac12 -b]\} \cup \{B_{a,V_{b,c}} \mid a \in [0,\tfrac12], c \in [0, \tfrac{b}{2}]\}\]
for fixed values $b \in \{ 0, \tfrac18, \tfrac14, \tfrac{7}{20}, \tfrac{7}{16}, \tfrac12 \}.$
\end{remark}

By the proof of Theorem~\ref{thm:main} and Proposition~\ref{prop:S_3}, every point on the surface $S_1$ is attained by a copula of the form $B_{a,A_r}$, whose mass distribution in the case $a=\tfrac{3}{20}$ and $r=\tfrac{5}{14}$ is shown in Figure~\ref{fig:S1}. This copula realizes the point $(\beta,\phi,\tau)=(\tfrac{3}{5},\tfrac{181}{400},\tfrac{27}{100})$.
Similarly, every point on the surface $S_2$ is attained by a copula of the form $B_{a,D_q}$, whose mass distribution in the case $a=\tfrac{3}{20}$ and $q=\tfrac{5}{28}$ is shown in Figure~\ref{fig:S2}. This copula realizes the point $(\beta,\phi,\tau)=(\tfrac{3}{5},\tfrac{181}{400},\tfrac{127}{200})$.

In the next two examples we also provide copulas realizing the points on the surface $S_4$ and triangle $T$.

\begin{example} \label{ex:S4}
    Let $(\beta, \phi, \tau) \ne (1,1,1)$ be a point on the surface $S_4$, i.e., $\beta \in [-1,1)$, $\phi = U(\beta)$ and $\tau \in [\ell(\phi),u(\phi)]$. Solving the equation $F_\beta(\alpha)=\phi$ gives $\alpha = \tfrac{1-\beta}{2}$. From definition \eqref{eq:primes} we obtain
    \[ \beta' = -1, \quad \phi' = L(\beta') = -\tfrac12, \quad \tau' = \frac{4(\tau-1)+(1-\beta)^2}{(1-\beta)^2} \in [\ell(\phi'),u(\phi')] = [-1,0].\]
    Furthermore,  $a = \tfrac{\beta+1}{4}$ and  $b=0$ by equation \eqref{eq:b}.
    The point $(\beta', \phi', \tau')$ lies on the line segment $P_4P_5$ and is realized by the copula 
    \[U_{0,c} = M\big(4, (\tfrac12 - c, \tfrac12, \tfrac12 + c), (4,3,2,1), (-1, 1, 1, -1)\big)\] 
   from Proposition~\ref{prop:S_3} for some $c \in [0, \tfrac12]$. It follows that $B_{a,U_{0,c}}$ realizes the point $(\beta,\phi,\tau)$. Figure~\ref{fig:S4} shows the mass distribution of copulas $U_{0,c}$ and $B_{a,U_{0,c}}$ in the case $a= \tfrac{3}{20}$ and $c = \tfrac17$. The copula $B_{3/20,U_{0,1/7}}$ realizes the point $(\beta,\phi,\tau)=(-\tfrac25,\tfrac{53}{200},\tfrac{3}{50})$.
\end{example}

\begin{figure}[ht]
\centering
\includegraphics{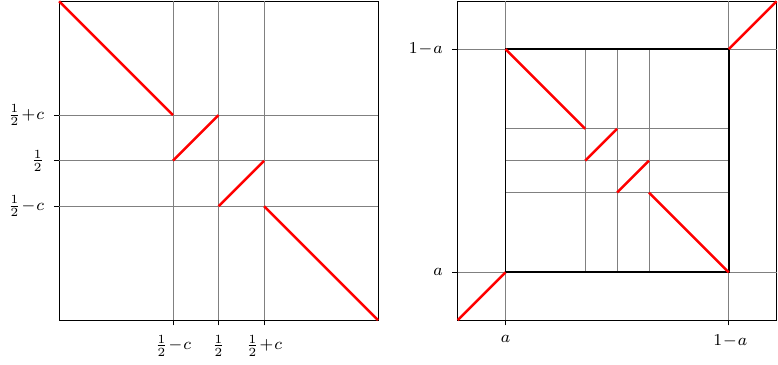}
    \caption{The mass distribution of copula $U_{0,c}$ (left) and copula $B_{a,U_{0,c}}$ (right) in the case $a=\tfrac{3}{20}$ and $c=\tfrac17$ from Example \ref{ex:S4}.}
    \label{fig:S4}
\end{figure}

\begin{example} \label{ex:T}
    Let $(\beta, \phi, \tau) \ne (1,1,1)$ be a point on the triangle $T=P_1P_2P_3$, i.e., $\beta =1$, $\phi \in [\tfrac14,1)$ and $\tau \in [\ell(\phi),u(\phi)]$. Solving the equation $F_\beta(\alpha)=\phi$ gives $\alpha = \tfrac{2}{\sqrt{3}}\sqrt{1-\phi}\in (0,1]$. From definition \eqref{eq:primes} we obtain
    \[ \beta' = 1, \quad \phi' = L(\beta') = \tfrac14, \quad \tau' = \frac{3\tau-4\phi+1}{4(1-\phi)} \in [\ell(\phi'),u(\phi')] = [0,\tfrac12].\]
    Furthermore, $a = \tfrac12 - \tfrac{1}{\sqrt{3}}\sqrt{1-\phi}$ and $b=\tfrac12$ by equation \eqref{eq:b}.
    The point $(\beta', \phi', \tau')$ lies on the line segment $P_1P_2$ and is realized by the copula 
    \[V_{1/2,c} = M\big(8, (c, \tfrac14, \tfrac14 + c, \tfrac12, \tfrac34 - c, \tfrac34, 1 - c), 
   (3,4,1,2,7,8,5,6), (1, -1, 1, -1, -1, 1, -1, 1)\big)\] 
   from Proposition~\ref{prop:S_3} for some $c \in [0, \tfrac14]$. It follows that $B_{a,V_{1/2,c}}$ realizes the point $(\beta,\phi,\tau)$. Figure~\ref{fig:T} shows the mass distribution of copulas $V_{1/2,c}$ and $B_{a,V_{1/2,c}}$ in the case $a= \tfrac{3}{20}$ and $c = \tfrac17$. The copula $B_{3/20,V_{1/2,1/7}}$ realizes the point $(\beta,\phi,\tau)=(1,\tfrac{253}{400},\tfrac{71}{100})$.
\end{example}

\begin{figure}[ht]
\centering
\includegraphics{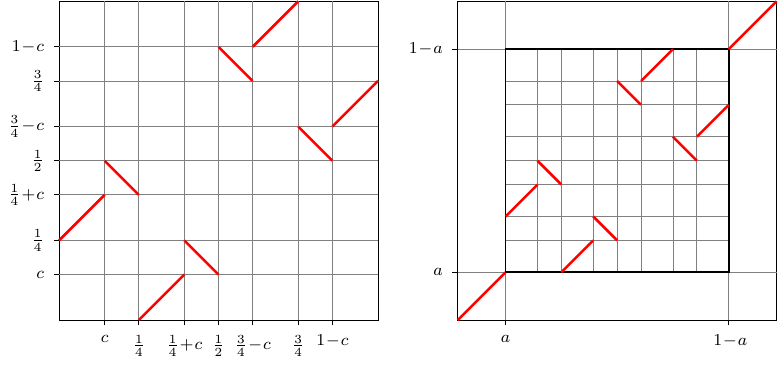}
    \caption{The mass distribution of copula $V_{1/2,c}$ (left) and copula $B_{a,V_{1/2,c}}$ (right) in the case $a=\tfrac{3}{20}$ and $c= \tfrac17$ from Example \ref{ex:T}.}
    \label{fig:T}
\end{figure}

Next, we explore further properties of the region $\Om$.

\begin{corollary}[Convexity and fixed-footrule sections]\label{cor:fixed-phi}
The region $\Om$ is compact and convex.
For $\phi\in[-\frac12,1]$, define
\[
\beta_-(\phi)\coloneqq 1-\sqrt{\tfrac83(1-\phi)},\qquad
\beta_+(\phi)\coloneqq \min\Bigl\{1,-1+\tfrac4{\sqrt3}\sqrt{\phi+\tfrac12}\Bigr\}.
\]
Then the section of $\Om$ at fixed value of Spearman's footrule $\phi$ is the possibly degenerate rectangle
\begin{equation}\label{eq:fixed-phi-section}
\bigl\{(\beta,\tau)\in\RR^2\mid (\beta,\phi,\tau)\in\Om\bigr\}
=[\beta_-(\phi),\beta_+(\phi)]\times[\ell(\phi),u(\phi)].
\end{equation}
Thus, once the value of Spearman's footrule is fixed, the admissible values of Kendall's tau and Blomqvist's beta vary independently over their respective intervals.
\end{corollary}

\begin{proof}
By Theorem~\ref{thm:main}, the region $\Om$ is closed and bounded, hence compact.
The function $L$ is convex, the function $U$ is concave, and $\ell$ and $u$ are affine functions.
Hence, the region $\Om$ is convex.

Now fix $\phi\in[-\frac12,1]$.
Since $1+\beta\ge0$ and $1-\beta\ge0$ for $\beta\in[-1,1]$, solving $L(\beta)\le \phi\le U(\beta)$ for $\beta\in[-1,1]$ gives
\[
\beta_-(\phi)\le \beta\le \beta_+(\phi),
\]
where the constraint $\beta\ge-1$ is automatically satisfied, since $\frac83(1-\phi)\le4$ and hence $\beta_-(\phi)\ge-1$.
The bounds $\ell(\phi)\le \tau\le u(\phi)$ do not involve $\beta$, which proves \eqref{eq:fixed-phi-section}.
\end{proof}

The projection of the region $\Om$ to the $(\beta,\tau)$-plane gives the exact $(\beta,\tau)$-region, which was already stated by Kruskal in terms of the quadrant measure \cite[Section~8]{Krus}; see also \cite[Exercise~5.17]{Nels} and \cite[Corollary~6]{KoBuKoMoOm19}.

\begin{corollary}[Exact $(\beta,\tau)$-region]\label{cor:beta-tau}
We have
\[
\bigl\{(\beta(C),\tau(C))\mid C\in\CC\bigr\}
=\Bigl\{(\beta,\tau)\in[-1,1]^2\mid \ \tfrac14(1+\beta)^2-1\le \tau\le 1-\tfrac14(1-\beta)^2\Bigr\}.
\]
\end{corollary}

\begin{proof}
By Theorem~\ref{thm:main}, for each fixed $\beta\in[-1,1]$, the projection of the $\beta$-section of $\Om$ onto the $(\phi,\tau)$-plane is 
\[
\Sigma_{\beta}\coloneqq
\bigl\{(\phi,\tau)\in\RR^2\mid \ L(\beta)\le \phi\le U(\beta),\ \ell(\phi)\le \tau\le u(\phi)\bigr\}.
\]
Since $\ell\le u$ are continuous, the projection of $\Sigma_\beta$ onto the $\tau$-axis is the interval
\[\bigcup_{\phi\in[L(\beta),U(\beta)]}[\ell(\phi),u(\phi)].\]
Since $\ell$ and $u$ are increasing, the left endpoint is $\ell(L(\beta))$ and the right endpoint is $u(U(\beta))$.
Using \eqref{eq:boundary-functions} gives
\[
\ell(L(\beta))=\tfrac14(1+\beta)^2-1,\qquad
u(U(\beta))=1-\tfrac14(1-\beta)^2,
\]
which proves the claim.
\end{proof}

\begin{remark}\label{rem:fibre}
The interval $[\ell(\phi),u(\phi)]$ of admissible values of $\tau$ over an admissible pair $(\beta,\phi)$ depends on $\phi$ alone, so Blomqvist's beta imposes no further restrictions on Kendall's tau once Spearman's footrule is known.
This complements the similarity-measure computations in \cite[Section~5]{KoBuSt}, by which knowing $\beta$ alone conveys, on average, very little information on the value of~$\tau$.
Similarly, once $\phi$ is known, calculating $\tau$ gives no additional information about $\beta$.
\end{remark}

\section{Concluding remarks}\label{sec:conclusion}

For every fixed $\beta\in[-1,1]$, the area of the $\beta$-section of $\Om$ in the $(\phi,\tau)$-coordinates is
\[
A(\beta)
\coloneqq \int_{L(\beta)}^{U(\beta)}\bigl(u(\phi)-\ell(\phi)\bigr)\de \phi
=\tfrac23\int_{L(\beta)}^{U(\beta)}(1-\phi)\de \phi
  =\tfrac{3}{256}(3-\beta)^2(1+\beta)(5-3\beta).
\]
Differentiating gives
\[
A'(\beta)=\tfrac{3}{64}(3-\beta)\bigl(3(\beta-1)^2-4\bigr).
\]
At $\beta=-1$ the section degenerates to a line segment and at $\beta=1$ it is the triangle $T$ with $A(1)=\frac{3}{16}$, while the sign of $A'$ shows that $A$ has the unique maximum $\frac14+\frac{\sqrt3}{6}$ at $\beta=1-\frac{2}{\sqrt3}$.
Integrating these sectional areas gives the volume of the region,
\[
\operatorname{vol}(\Om)
=\int_{-1}^{1}A(\beta)\de \beta
=\tfrac{3}{256}\int_{-1}^{1}
\bigl(-3\beta^4+20\beta^3-34\beta^2-12\beta+45\bigr)\de \beta
=\tfrac{31}{40},
\]
which is $\frac{31}{240}\approx0.129$ of the volume of the bounding box $[-1,1]\times[-\frac12,1]\times[-1,1]$.
For comparison, if we denote Gini's gamma by $\gamma$, then $\operatorname{vol}(\Omega_{\phi,\gamma,\tau})=\frac{3}{16}$, which is $\frac{1}{32}\approx 0.031$ of the volume of the bounding box $[-\frac12,1]\times[-1,1]\times[-1,1]$, and $\operatorname{vol}(\Omega_{\beta,\phi,\gamma})=\frac{19}{40}$, which is $\frac{19}{240} \approx 0.079$ of the volume of the bounding box $[-1,1]\times[-\frac12,1]\times[-1,1]$, see \cite{KoBuLaMoSt} and \cite{KoBuMo26}, respectively.
This relatively large volume of $\Om$ reflects Remark~\ref{rem:fibre}: the three coefficients $\beta$, $\phi$ and $\tau$ constrain one another only through the two bivariate regions $\Omega_{\beta,\phi}$ and $\Omega_{\phi,\tau}$.

In view of \cite{KoBuLaMoSt,KoBuMo26} and Theorem~\ref{thm:main}, the next natural problems are the exact region $\Omega_{\beta,\gamma,\tau}$ and the full four-dimensional region determined by $\beta$, $\phi$, $\gamma$ and $\tau$.
To the best of our knowledge, both are still open.
Exact regions involving Spearman's rho also appear to be considerably harder to determine already for pairs of coefficients, cf.~\cite{ScPaTr,KoBuSt2,TsScTr,AnRo26,AnRoSt26}. 

\section*{Acknowledgments}

The authors acknowledge financial support from the Slovenian Research and Innovation Agency (ARIS), research core funding No. P1-0222 (Damjana Kokol Bukov\v{s}ek, Bla\v{z} Moj\v{s}kerc, Nik Stopar), No. P5-0161 (Liljana Ferbar Tratar) and project J1-70034 (Damjana Kokol Bukov\v{s}ek, Bla\v{z} Moj\v{s}kerc, Nik Stopar).

\section*{Use of AI tools}

During the preparation of this work, the authors used ChatGPT (OpenAI) as a writing assistant. The authors edited the paragraphs as needed and take full responsibility for the content of the paper.

\bibliographystyle{amsplain}
\bibliography{biblio}

\end{document}